\documentclass[12pt]{amsart}
\usepackage{amssymb}
\usepackage{mathtools}
\usepackage[english, french]{babel}
\usepackage{epsfig}
\usepackage{mathptmx}
\usepackage{amsmath,amsfonts,amssymb}
\usepackage{mathtools}
\usepackage{mathrsfs}
\usepackage[all]{xy}
\usepackage{graphicx}
\usepackage{latexsym}
\usepackage{verbatim}

\numberwithin{equation}{section}

\def\Re{{\sf Re}\,}
\def\Im{{\sf Im}\,}

\newcommand{\D}{\mathbb D}

\newcommand{\R}{\mathbb R}
\newcommand{\Ha}{\mathbb H}

\newcommand{\C}{\mathbb C}

\newcommand{\N}{\mathbb N}

\def\Re{{\sf Re}\,}
\def\Im{{\sf Im}\,}

\def\Re{{\sf Re}\,}
\def\Im{{\sf Im}\,}

\def\Re{{\sf Re}\,}
\def\Im{{\sf Im}\,}

\def\1#1{\overline{#1}}
\def\2#1{\widetilde{#1}}
\def\3#1{\widehat{#1}}
\def\4#1{\mathbb{#1}}
\def\5#1{\frak{#1}}
\def\6#1{{\mathcal{#1}}}

\def\Re{{\sf Re}\,}
\def\Im{{\sf Im}\,}

\newcommand{\mcite}[1]{\csname b@#1\endcsname}

\theoremstyle{theorem}

\def\Re{{\sf Re}\,}
\def\Im{{\sf Im}\,}

\newtheorem{theorem}{Theorem}[section]
\newtheorem{lemma}[theorem]{Lemma}
\newtheorem{proposition}[theorem]{Proposition}
\newtheorem{corollary}[theorem]{Corollary}

\theoremstyle{definition}
\newtheorem{definition}[theorem]{Definition}

\theoremstyle{remark}
\newtheorem{remark}[theorem]{Remark}

\numberwithin{equation}{section}

\title[Semigroup-fication]{Semigroup-fication of univalent self-maps  of the unit disc}

\author[F. Bracci]{Filippo Bracci$^\dag$}
\address{F. Bracci: Dipartimento di Matematica, Universit\`a di Roma ``Tor Vergata", Via della Ricerca
Scientifica 1, 00133, Roma, Italia.} \email{fbracci@mat.uniroma2.it}

\author[O. Roth]{Oliver Roth}
\address{O. Roth: Department of Mathematics, University of W\"urzburg, Emil Fischer Strasse 40, 97074, W\"urzburg, Germany.} \email{roth@mathematik.uni-wuerzburg.de}


\thanks{$^\dag\,$Partially supported by PRIN 2017 {\sl Real and Complex Manifolds: Topology, Geometry and holomorphic dynamics}, Ref: 2017JZ2SW5, by GNSAGA of INdAM and  by the MIUR Excellence Department Project awarded to the
Department of Mathematics, University of Rome Tor Vergata, CUP E83C18000100006}

\long\def\REM#1{\relax}

\begin{document}
\maketitle

\selectlanguage{english}
\begin{abstract}
Let $f$ be a univalent self-map of the unit disc. We introduce a technique,
that we call {\sl semigroup-fication}, which allows to construct a continuous
semigroup $(\phi_t)$ of holomorphic self-maps of the unit disc whose time one
map $\phi_1$ is, in a sense, very close to $f$. The semigrup-fication of $f$ is of the same type as $f$ (elliptic, hyperbolic, parabolic of positive step or parabolic of zero step) and there is a one-to-one correspondence between the set of boundary regular fixed points of $f$ with a given multiplier and the corresponding set for $\phi_1$. Moreover, in case $f$ (and hence $\phi_1$) has no interior fixed points, the slope of the orbits converging to the Denjoy-Wolff point is the same.  The construction is based on holomorphic models, localization techniques and Gromov hyperbolicity. As an application of this construction, we  prove that in the non-elliptic case, the orbits of $f$ converge non-tangentially   to the Denjoy-Wolff point if and only if the Koenigs domain of $f$ is ``almost symmetric'' with respect to vertical lines.
\end{abstract}

\tableofcontents

\section{Introduction}
Iteration theory in the unit disc and more generally on complex manifolds has been a subject of study for more than one century, starting from the work of Schr\"oder in the 1870's and Koenigs in the 1880's. We refer the reader to the book of M. Abate \cite{Abate} for history and complete bibliography on the subject  (see also \cite{BCDbook} and \cite{CoMac, Sha}, where iteration theory in the unit disc is developed for applications to composition operators).

One of the most striking results in iteration theory  is that every
holomorphic self-map of the unit disc  $\D:=\{\zeta\in \C: |\zeta|<1\}$ can be
``linearized''. To be precise,  let  $f:\D \to \D$ be holomorphic. The Schwarz
lemma implies that  $f$ is either the identity map or has at most one fixed point in
$\D$. In case $f$ has no fixed points in $\D$ or  a fixed point $z_0\in
\D$ such that $f'(z_0)\neq 0$,   there exists a holomorphic function $h:\D \to
\C$ such that $h\circ f=\psi\circ h$, where $\psi$ is a suitable automorphism
of $\C$. This linearization model was developed over a period of many decades 
starting from  Koenigs~\cite{Koe}, Valiron~\cite{Va}, Pommerenke~\cite{Po}, Baker
and Pommerenke~\cite{BaPo}, Cowen~\cite{Co} and most recently Arosio and the
first named author~\cite{AB}. 
It breaks into three main cases, called elliptic,
hyperbolic and parabolic.

If $f$ has a fixed point $z_0\in \D$---in this case $f$ is called {\sl
  elliptic}---one can choose $\psi(z)=\lambda z$  (with $\lambda=f'(z_0)$) and
$\bigcup_{n\leq 0}\lambda^n h(\D)$ equal either to $\D$ (if $f$ is an
automorphism of $\D$) or equal to $\C$. In case $f$ is not an automorphism of $\D$, then the sequence of iterates of $f$, $\{f^{\circ n}(z)\}$, converges to $z_0$.

In case $f$ has no fixed point in $\D$, there exists a unique point $\tau\in
\partial \D$, called the {\sl Denjoy-Wolff point} of $f$, such that
$\{f^{\circ n}(z)\}$  converges to $\tau$ for all $z\in \D$. In addition, the
non-tangential limit of $f$ at $\tau$ is $\tau$, {\sl i.e.},
$\angle\lim_{z\to\tau}f(z)=\tau$, and $\angle\lim_{z\to\tau}f'(z)=\alpha\in
(0,1]$. The map $f$ is called {\sl hyperbolic} if $\alpha<1$, and {\sl parabolic} otherwise. In both of these cases, 
one can choose  $\psi(z)=z+i$, and $\bigcup_{n\leq 0}(h(\D)-in)$ equal to
either a strip $\{w\in \C: 0<\Re w<a\}$ for some $a>0$, or $\Ha:=\{w\in \C:
\Re w>0\}$, or $-\Ha$ or $\C$. The set $\bigcup_{n\leq 0}(h(\D)-in)$ is a
strip if and only if $f$ is hyperbolic. The parabolic case breaks into two subcases.
If  $\bigcup_{n\leq 0}(h(\D)-in)=\Ha$ (or $-\Ha$), then $f$ is called {\sl parabolic of positive hyperbolic step}. If $\bigcup_{n\leq 0}(h(\D)-in)=\C$, then $f$ is called  {\sl parabolic of zero hyperbolic step}.

If $f$ is univalent ({\sl i.e.}, holomorphic and injective), the holomorphic function $h$ which realizes the previous linearization model can be chosen to be univalent as well and is then unique up to post-composition with affine transformations.
In this case, the map  $h$ is called the {\sl Koenigs function of} $f$ and its
image domain $\Omega:=h(\D)$ the {\sl Koenigs domain of} $f$.  Every other
linearization of $f$ factorizes through $h$. Since $f=h^{-1}\circ \psi \circ
h$ and $\psi$ is affine, it is thus the geometry of the  Koenigs
 domain which encapsulates the  dynamical properties of the map $f$. In
 particular, it determines the way the iterates of $f$  converge to the
 Denjoy--Wolff point of $f$.
 
One of the mains aim of this paper is to introduce a technique, that we call {\sl semigroup-fication} of an univalent self-map $f$, which allows to define a continuous semigroup of holomorphic self-maps of $\D$ which is, in some sense, very close to $f$. We briefly describe the semigroup-fication technique in case $f$ is non-elliptic. If $f:\D \to \D$ is univalent with no fixed points in $\D$ and $h$ is its Koenigs function, then the Koenigs domain $\Omega:=h(\D)$ of $f$ is  \textit{asymptotically starlike at infinity}, that is, $\Omega+i\subseteq \Omega$ and $\bigcup_{n\in \N}(\Omega-in)$ is either a strip,  a half-plane or $\C$. We   define the set $\Omega^\ast\subseteq \Omega$ by considering the union of all $z\in \Omega$ such that $z+it\in \Omega$ for all $t>0$. Hence, by construction,  $\Omega^*$ is  \textit{starlike at infinity}, so 
we call $\Omega^\ast$ the {\sl starlike-fication of $\Omega$} (see
Section~\ref{Sec:star}). It is not  difficult to show that  $\Omega^\ast$ is a
non-empty, open, connected and simply connected set  (Lemma \ref{Lem:starlike-fication}). Hence, by the Riemann mapping theorem, there is  a Riemann map  $h^\ast:\D \to \Omega^\ast$, and setting  $\phi_t(z):=(h^\ast)^{-1}(h^\ast(z)+it)$, we obtain a continuous semigroup $(\phi_t)$ of holomorphic self-maps of $\D$ ({\sl i.e.}, the flow of a semicomplete holomorphic vector field in $\D$). A similar construction can be done in case $f$ is elliptic, using the invariance of the Koenigs domain under the map $z\mapsto \lambda z$. The semigroup $(\phi_t)$ is the {\sl semigroup-fication} of $f$. 

As a matter of notation, if $f:\D\to \D$ is holomorphic and $A\geq 1$, we let 
\[
\hbox{Fix}_A(f):=\{\sigma\in \partial \D: \angle\lim_{z\to \sigma}f(z)=\sigma, \angle\lim_{z\to \sigma}f'(z)=A\},
\]
be the set of {\sl boundary regular fixed points} of $f$.

The main result about semigroup-fication is the following:

\begin{theorem}\label{Thm:semigroupfication-intro}
Let $f:\D \to \D$ be univalent and let $(\phi_t)$ be the  semigroup-fication of $f$. Then
\begin{enumerate}
\item  $f$ is of the same type (elliptic, hyperbolic, parabolic of positive hyperbolic step or parabolic of zero hyperbolic step) as $\phi_1$. 
\item If $f$--and hence $\phi_1$---is elliptic, then $f(z_0)=z_0$ and $\phi_1(z_0)=z_0$ for some $z_0\in \D$ and $f'(z_0)=\phi_1'(z_0)$, 
\item If $f$--and hence $\phi_1$---is non-elliptic, with dilation coefficient $\alpha\in (0,1]$ at its Denjoy-Wolff point, then $\phi_1$ has dilation coefficient $\alpha\in (0,1]$ at its Denjoy-Wolff point.
\item for every $A\geq 1$ there is a one-to-one correspondence between
  $\hbox{\rm Fix}_A(f)$ and $\hbox{\rm Fix}_A(\phi_1)$.
\item If $f$---and hence $\phi_1$---is non-elliptic, and  $\{n_k\}\subset\N$ is an increasing sequence converging to $\infty$, then the following are equivalent:
\begin{itemize}
\item[a)] $\{f^{\circ n_k}(z)\}$ converges non-tangentially to the Denjoy-Wolff point of $f$ for some---and hence any---$z\in \D$,
\item[b)] $\{\phi_{n_k}(z)\}$ converges non-tangentially to the Denjoy-Wolff point of $(\phi_t)$  for some---and hence any---$z\in \D$.
\end{itemize}
\end{enumerate}
\end{theorem}

In case $f$ is hyperbolic (but not necessarily univalent), it is known (see
\cite{Va, Co}) that $\{f^{\circ n}(z)\}$ converges non-tangentially to the
Wolff-Denjoy point $\tau$ of $f$  for all $z\in \D$.
If  $f$ is parabolic of positive hyperbolic step, $\{f^{\circ n}(z)\}$
converges tangentially to $\tau$ (see \cite{Po}). If $f$ is parabolic of zero hyperbolic step, then simple examples show that both types of convergence can occur.

As an application of our semigroup-fication technique, in this paper we completely characterize the way $\{f^{\circ n}(z)\}$ converges  to the
Denjoy-Wolff point  in terms of the Euclidean shape of the Koenigs domain of $f$ 
in the case of univalent self-maps $f$ of $\D$. 

In particular, we show that the non-tangential convergence of 
 $\{f^{\circ n}(z)\}$ is equivalent to a simple geometric condition on the Koenigs domain  when $f$ is univalent parabolic of zero hyperbolic step. Our argument also leads to another proof of the tangential convergence in case of univalent parabolic maps with positive hyperbolic step (\cite{Po},  \cite{BaPo}) and of the non-tangential convergence in case of hyperbolic univalent  maps (\cite{Va}).

In order to state our result, we introduce a notation: if $\Omega\subsetneq \C$ is a domain, for  $p\in \C$ and $t>0$, we let
\[
\tilde\delta_{\Omega,p}^+(t):=\inf\{|z-(p+it)|: \Re z\geq \Re p, z\in \C\setminus \Omega\},
\] 
\[
\tilde\delta_{\Omega,p}^-(t):=\inf\{|z-(p+it)|: \Re z\leq \Re p, z\in \C\setminus \Omega\}.
\] 
We also let $\delta_{\Omega,p}^\pm(t):=\min\{t, \tilde\delta_{\Omega,p}^\pm(t)\}$.

\begin{theorem}\label{Thm:main}
Let $f:\D \to \D$ be univalent without fixed points in $\D$ and  Denjoy-Wolff point $\tau\in\partial \D$, and let $\Omega$ be the Koenigs domain of $f$. Suppose that $\{n_k\}\subset \N:=\{0,1,2,\ldots\}$ is an increasing sequence converging to $+\infty$. Then the following are equivalent:
\begin{enumerate}
\item For some---and hence any---$z\in \D$, $\{f^{\circ n_k}(z)\}$ converges non-tangentially to $\tau$.
\item For some---and hence any---$p\in \Omega$ 
\[
0<\liminf_{k\to+\infty}\frac{\delta^+_{\Omega, p}(n_k)}{\delta^-_{\Omega, p}(n_k)}\leq \limsup_{k\to+\infty}\frac{\delta^+_{\Omega, p}(n_k)}{\delta^-_{\Omega, p}(n_k)}<+\infty.
\]
\end{enumerate}
\end{theorem}

The previous result was proven by the first named author together with M. Contreras, S. D\'iaz-Madrigal, H. Gaussier and A. Zimmer in case $f$ is the time one flow of a semicomplete holomorphic vector field of the unit disc (see \cite[Theorem 1.1]{BCDGZ}, see also \cite[Chapter 17]{BCDbook}). Since time one flows are univalent, but in general univalent self-maps of the unit disc can not be embedded into a flow of a semicomplete holomorphic vector field, Theorem~\ref{Thm:main} generalizes \cite[Theorem 1.1]{BCDGZ}. 

Theorem~\ref{Thm:main} follows from the semigroup-fication technique which, thanks to Theorem~\ref{Thm:semigroupfication-intro}.(5) allows to reduce  Theorem~\ref{Thm:main} to the corresponding result \cite[Theorem 1.1]{BCDGZ} for flows of semicomplete holomorphic vector fields (see Section~\ref{Sec:proof} for details). 

The proofs of our results are mainly based on localization results for the
hyperbolic metric and the hyperbolic distance and, in particular,  on Gromov's
hyperbolicity theory. In  order to make the paper self-contained, we recall
and partly prove the results of hyperbolicity theory we need in
Section~\ref{Sec:hyperbolicity}. In Section \ref{sec:canonical} and Section
\ref{Sec:semigroup} we briefly review the relevant facts about canonical
models for iteration and semigroups of holomorphic self-maps. In Section \ref{Sec:BRFP} we discuss  boundary regular fixed points and their characterization through models. In Section~\ref{Sec:star} we
describe the construction and the properties of the starlike-fication
$\Omega^*$ of a domain $\Omega$ which is  asymptotically starlike
at infinity.
The basic technical result of the paper is proved in
Section~\ref{Sec:Gromov}. It asserts that if $\Omega$ is  an asymptotically
starlike domain of parabolic type which is Gromov hyperbolic (but not
necessarily simply connected), then there exists a Lipschitz curve in its
starlike-fication $\Omega^\ast$ ``escaping  to $\infty$'' which is a quasi-geodesic  both in $\Omega$ and $\Omega^\ast$ (see Theorem~\ref{Thm:sigma-qg}). 
Finally, the proof of Theorem~\ref{Thm:semigroupfication-intro} and the proof
of Theorem~\ref{Thm:main} are given  in Section~\ref{Sec:proof}.

\section{Right and left distance of a domain}

For a given set $\Omega\subsetneq \C$, we let
\[
\delta_\Omega(z):=\inf\{|z-p|: p\in \C\setminus\Omega\} \, , \qquad z \in \Omega.
\]
For  $p\in \C$ and $t>0$, we define
\[
\tilde\delta_{\Omega,p}^+(t):=\inf\{|z-(p+it)|: \Re z\geq \Re p, z\in \C\setminus \Omega\},
\] 
\[
\tilde\delta_{\Omega,p}^-(t):=\inf\{|z-(p+it)|: \Re z\leq \Re p, z\in \C\setminus \Omega\}.
\] 
Note that $\tilde\delta^+_{\Omega, p}(t)=\delta_{\Omega^+}(p+it)$, where $\Omega^+=\Omega\cup\{w\in \C: \Re w<\Re p\}$.

Note also that if $p+it\in \C\setminus \Omega$, then $\tilde\delta_{\Omega,p}^+(t)=\tilde\delta_{\Omega,p}^-(t)=0$, while for $p+it\in \Omega$, 
\[
\delta_\Omega(p+it)=\min\{\tilde\delta_{\Omega,p}^+(t), \tilde\delta_{\Omega,p}^-(t)\}.
\] 

Moreover, for $t>0$ we let
\[
\delta^+_{\Omega, p}(t):=\min\{\tilde\delta_{\Omega,p}^+(t), t\},\quad \delta^-_{\Omega,p}(t):=\min\{\tilde\delta_{\Omega,p}^-(t), t\}.
\]
If $\Omega$ is starlike at infinity ({\sl i.e.}, $\Omega+it\subseteq \Omega$ for all $t\geq 0$), then $(0,+\infty)\ni t\mapsto \delta_{\Omega,p}^{\pm}(t)$ is non-decreasing. 

Simple geometric considerations allow to prove the following lemma:

\begin{lemma}\label{Lem:same-delta}
Let $\Omega\subsetneq \C$ be a   domain such that for every $z\in \Omega$ there exists $t_z\in \R$ such that $z+it\in \Omega$ for all $t>t_z$. Then for all $p, q\in \Omega$ there exist $0<c<C$  such that for all $t>0$ 
\[
c \delta^{\pm}_{\Omega,p}(t)\leq \delta^{\pm}_{\Omega,q}(t)\leq C\delta^{\pm}_{\Omega,p}(t).
\]
\end{lemma}

\section{Hyperbolic geometry}\label{Sec:hyperbolicity}

In this section we recall the notions and results of hyperbolic geometry we need in the paper. We refer the reader to \cite{Abate, Buck, BCDbook, Koba, Ghys} for details.

Let $D\subset \C$  be a domain, $z\in D$ and $v\in \C$. The {\sl hyperbolic norm} of $v$ at $z$ in $D$ is
\[
\kappa_D(z;v):=\inf\left\{\frac{|v|}{|f'(0)|} \, \bigg| \,  f:\D \to D \hbox{\ is holomorphic}, f(0)=z\right\}.
\]
By Schwarz's Lemma it follows immediately that if $D\subsetneq \C$ is a simply connected domain then for all $z\in D$ and $v\in \C$ we have $\kappa_D(z;v)=\frac{|v|}{f'(0)}$, where $f:\D\to D$ is the Riemann map such that $f(0)=z$ and $f'(0)>0$. 

The {\sl hyperbolic distance} between $z,w\in D$ is defined as 
\[
k_D(z,w):=\inf \int_0^1 \kappa_D(\gamma(\tau);\gamma'(\tau))d\tau,
\]
where the infimum is taken over all piecewise $C^1$-smooth curves $\gamma:[0,1]\to D$ such that $\gamma(0)=z$ and $\gamma(1)=w$. 

As a consequence of Schwarz's Lemma, every biholomorphism between two domains is an isometry for the hyperbolic norm and distance, while every holomorphic function does not expand   the hyperbolic norm and distance. 

It follows from the uniformization theorem  that $k_D(z,w)=0$ for some---and hence for all---$z, w\in D$ if and only if $D=\C$ or $D=\C\setminus\{p\}$ for some $p\in \C$. In all other cases,  namely,  if $D$ is {\sl  hyperbolic} ({\sl i.e.}, holomorphically covered by $\D$),  $(D, k_D)$ is a complete metric space. Note that this implies in particular that if $D\subset \C$ is a domain whose complement contains more than one point, then $\lim_{n\to \infty}k_D (z_n, z_0)=+\infty$ for all fixed $z_0\in D$ and $\{z_n\}\subset D$ such that $\{z_n\}$ lies eventually outside any compacta of $D$. 

The {\sl hyperbolic length} of an absolutely continuous curve  $\gamma:[s,t]\to D$  is  
\[
\ell_D(\gamma;[s,t])=\int_s^t \kappa_D(\gamma(\tau);\gamma'(\tau))d\tau.
\]

 An absolutely continuous curve $\eta: I\to D$ defined on an interval $I\subset \R$ is a  {\sl geodesic} if 
\[
\ell_D(\eta;[s,t])=k_D(\eta(s), \eta(t))
\]
 for all $s \le t$ belonging to $I$.
 
 Since $(D,k_D)$ is complete, it follows from the Hopf-Rinow theorem that for every $z, w\in D$ there exists a geodesic $\eta:[0,1]\to D$ in $D$ such that $\eta(0)=z$ and $\eta(1)=w$.
 Geodesics joining two different points might not be unique (up to parameterization) in general, but, as a consequence of the Riemann mapping theorem and a direct inspection in the unit disc, if $D\subsetneq \C$ is simply connected  then every two points of $D$ can be joined by a unique (up to parameterization) geodesic.

Three geodesics $\gamma_1, \gamma_2, \gamma_3:[0,1]\to D$ such that $\gamma_1(0)=\gamma_2(0)$, $\gamma_1(1)=\gamma_3(0)$ and $\gamma_2(1)=\gamma_3(1)$ form a {\sl geodesic triangle} in $D$. The sets $\gamma_j([0,1])$, $j=1,2,3$, are called the {\sl edges} of the geodesic triangle.

Given a hyperbolic domain $D\subsetneq \C$, we say that $(D, k_D)$ is {\sl Gromov hyperbolic} if there exists a constant $G>0$, called the {\sl Gromov constant} of $(D, k_D)$, such that for every geodesic triangle $\{\gamma_1, \gamma_2,\gamma_3\}$, each point of each edge  stays at hyperbolic distance no more than $G$ from the other two edges of the geodesic triangle.
This is the well known Rips thin triangle definition of Gromov hyperbolicity for geodesic metric spaces such as $(D,k_D)$.
For simplicity, we say that a hyperbolic domain $D\subsetneq \C$ is Gromov hyperbolic, if $(D,k_D)$ is
Gromov hyperbolic.

Since every simply connected domain $D\subsetneq \C$ is biholomorphic to $\D$,
hence ($D,k_D)$ is isometric to $(\D,k_{\D})$, and the unit disc  is well known
to be Gromov hyperbolic, it follows that  every simply connected domain
$D\subsetneq \C$ is Gromov hyperbolic (with the same  Gromov constant as  the unit disc).

 \begin{definition}
Let $I \subset \R$ be an interval, $D\subsetneq \C$  a hyperbolic domain and
$\gamma: I\to D$ an absolutely continuous curve. Let $A\geq 1$, $B\geq
0$. We say that $\gamma$ is a {\sl $(A,B)$-quasi-geodesic} if   for all $s\leq t$ belonging to $I$,
\begin{align*}
\ell_D(\gamma;[s,t]) \leq Ak_D(\gamma(s), \gamma(t))+B.
\end{align*}
We say that $\gamma$ is a {\sl  quasi-geodesic} if there exist $A\geq 1, B\geq 0$ such that $\gamma$ is  a  $(A,B)$-quasi-geodesic.
\end{definition}

We sometimes say that $\gamma : I \to D$ is a (quasi-)geodesic \textit{in $D$}
when we need to emphasize the ambient space $(D,k_D)$. If this ambient space is Gromov
hyperbolic, then every quasi-geodesic is ``shadowed' by a geodesic. This is
the content of Gromov's shadowing lemma. It says that for any
$A\geq 1$ and $B\geq 0$ there exists $M>0$ (which depends only on $A, B$ and
the Gromov's constant of $D$) such that if $\gamma:[a,b]\to D$ is a
$(A,B)$-quasi-geodesic, then  there exists a geodesic $\eta:[0,1]\to D$ such that $\eta(0)=\gamma(a)$, $\eta(1)=\gamma(b)$ and
for every $t\in [a,b]$, $s\in [0,1]$
\begin{equation}\label{Eq:shadows}
k_D(\gamma(t), \eta([0,1]))<M, \quad k_D(\eta(s), \gamma([a,b]))<M.
\end{equation}
See, for instance, \cite[Theorem 6.3.8]{BCDbook}) for a proof of the shadowing lemma.

In this paper, quasi-geodesics  play a significant role,
in particular since they are very useful for detecting non-tangential convergence.
The proof of the following result is based on Gromov's shadowing lemma and can
be found e.g.~in \cite[Proposition 4.5]{BCDG} (or \cite[Corollary 6.3.9]{BCDbook}).

\begin{proposition} \label{non-tg-sector-hyp}
Let $D\subsetneq \C$ be a simply connected domain and let $f:\D\to D$ be a Riemann map. Suppose  $\eta:[0,+\infty)\to D$ is a quasi-geodesic such that $\lim_{t\to+\infty}k_D(\eta(0),\eta(t))=+\infty$. Then
\begin{enumerate}
\item  there exists $\tau\in\partial \D$ such that $\lim_{t\to+\infty}f^{-1}(\eta(t))=\tau$,
\item a sequence $\{z_n\} \subset \D$ converges non-tangentially to $\tau$ if and only if there exists $C>0$ such that for all $n\in \N$,
\[
k_D(f(z_n), \eta([0,+\infty))=\inf\{k_D(f(z_n), \eta(t)): t\in [0,+\infty)\}\leq C.
\]
\end{enumerate}
\end{proposition}

If $D\subsetneq \C$ is a domain, and $z\in D$, $R>0$, we let
\[
B_D(z,R):=\{w\in D: k_D(z,w)<R\}.
\]
We  need the following localization lemma:

\begin{lemma}\label{lem:local-disc}
Let $D\subsetneq \C$ be a hyperbolic domain. Then for every $R>0$ there exists $c>1$ such that for all $z\in D$ and $v\in \C$,
\[
\kappa_D(z; v)\leq \kappa_{B_D(z,R)}(z;v) \leq c \kappa_D(z; v).
\]
\end{lemma}

One can take $c=\cosh(R)$, so $1/c$ is the \textit{euclidean} radius of $B_{\D}(0,R)$, the disc in $\D$ centered at the origin of hyperbolic radius $R$.
\begin{proof}
Since $B_D(z,R)\subset D$ and $B_D(z,R)\ni z\mapsto z\in D$ is holomorphic, the first estimate follows  from the non-increasing property of the hyperbolic norm under holomorphic maps. 

In order to prove the second estimate, it is clear from the definition of hyperbolic norm that it is enough to consider $v=1$.
Let $c:=\kappa_{B_{\D}(0,R)}(0;1)$. Fix $z \in D$ and let $f : \D \to D$ be a holomorphic function such that $f(0)=z$. Since $f : (\D,k_{\D}) \to (D,k_D)$ is distance non-increasing, it maps $B_{\D}(0,R)$ into $B_{D}(z,R)$  again in a nonexpanding way, so
$$ \kappa_{B_{D}(z,R)}(f(\xi);f'(\xi)) \le \kappa_{B_{\D}(0,R)}(\xi;1)\, , \qquad \xi \in B_{\D}(0,R) \, .$$
Setting $\xi=0$, we get $\kappa_{B_D(z,R)}(z;1) \le c/|f'(0)|$, which proves the second estimate.
%
%
\end{proof}

\begin{remark}\label{disc-total-geo}
Let $D\subsetneq \C$ be a simply connected domain,  $z\in D$ and $R>0$. Then $B_D(z,R)$ is {\sl totally geodesic} in $D$. Namely, for every $p,q\in B_D(z,R)$, the geodesic in $D$ joining $p$ and $q$ is contained in $B_D(z,R)$. This follows easily since $D$ is isometric to $\D$ via a Riemann map, and, since the group of automorphisms of $\D$ act transitively on $\D$, it is enough to check the statement for hyperbolic discs centered at the origin.
\end{remark}

In case $D\subset\C$ is hyperbolic but not necessarily simply connected, the previous remark can be replaced by the following

\begin{lemma}\label{Lem:disc-quasi-total}
Let $\Omega\subset\C$ be a hyperbolic domain. Let $z_0\in D$ and $T>0$. Then for every $z,w\in B_D(z_0, T)$, any  geodesic $\gamma:[0,1]\to D$ in $D$ such that $\gamma(0)=z$, $\gamma(1)=w$ satisfies $\gamma(t)\in B_D(z_0, 2T)$ for all $t\in [0,1]$.
\end{lemma}
\begin{proof}
Let $z,w\in B_D(z_0, T)$. By the triangle inequality,
\[
k_D(z,w)\leq k_D(z,z_0)+k_D(w,z_0)<2T.
\]
Let $\gamma:[0,1]\to D$ be a geodesic such that $\gamma(0)=z$ and $\gamma(1)=w$ and assume by contradiction that there exists $t_0\in (0,1)$ such that $\gamma(t_0)\not\in B_D(z_0, 2T)$. Thus, by continuity, we can find $0<s_0<t_0<s_1<1$ such that $\gamma(s_0), \gamma(s_1)\in \partial  B_D(z_0, T)$. Since it is clear that 
\[
k_D(\partial  B_D(z_0, T), \partial  B_D(z_0, 2T))=T,
\]
we have
\[
2T>k_D(z,w)>k_D(\gamma(s_0), \gamma(s_1))=k_D(\gamma(s_0), \gamma(t_0))+k_D(\gamma(t_0), \gamma(s_1))\geq 2T,
\]
a contradiction.
\end{proof}

\section{Canonical models for iteration}\label{sec:canonical}

Let $f:\D \to \D$ be a holomorphic map without fixed points in $\D$. The well known Denjoy-Wolff Theorem (see, {\sl e.g.}, \cite{Abate, BCDbook}) implies that there exists a unique point $\tau\in\partial \D$ such that the sequence $\{f^{\circ n}\}$ of iterates of $f$ converges uniformly on compacta to the constant map $z\mapsto \tau$. Moreover, there exists $\alpha\in (0,1]$ such that
\[
\angle\lim_{z\to \tau}f'(z)=\alpha,
\]
where $\angle\lim$ denotes the non-tangential limit. The map $f$ is called {\sl hyperbolic} if $\alpha<1$ and {\sl parabolic} if $\alpha=1$.

We state the following linearization result for univalent ({\sl i.e.}, injective and holomorphic) maps. The theorem has a long history, starting with Valiron~\cite{Va}, Pommerenke~\cite{Po}, Baker and Pommerenke~\cite{BaPo}, Cowen~\cite{Co}, and Bourdon and Shapiro~\cite{BoSha}. The statement here comes from \cite{AB}.

\begin{theorem}\label{Thm:canonical-model}
Let $f:\D \to \D$ be univalent and with no fixed points in $\D$. Then there exists $h:\D \to \C$ univalent such that
\begin{enumerate}
\item $h(f(z))=h(z)+i$ for all $z\in \D$.
\item $\bigcup_{m\in \N}(h(\D)-im)=\Lambda$, where either $\Lambda=\{w\in \C: 0<\Re w<a\}$ for some $a>0$, or $\Lambda=\Ha:=\{w\in \C: \Re w>0\}$, or $\Lambda=-\Ha$ or $\Lambda=\C$.
\item If $g:\D \to \C$ is holomorphic and $g(f(z))=g(z)+i$ for all $z\in \D$ then there exists a surjective, holomorphic map $\psi:\Lambda\to \bigcup_{n\geq 0}(g(\D)-ni)$ such that $g=\psi\circ h$ and $k(z+i)=k(z)+i$ for all $z\in \Lambda$. Moreover, if $g$ is univalent, then $\psi$ is a biholomorphism. 
\item $\lim_{n\to \infty}\frac{k_\Lambda(w, w+in)}{n}=-\frac{1}{2}\log \alpha$ for all $w\in \Lambda$.
\item $\lim_{n\to \infty}k_\D(f^{\circ n}(z), f^{\circ n}(w))=k_\Lambda(h(z), h(w))$ for all $z,w\in \D$.
\end{enumerate}
\end{theorem}

The map $h$ is  called the {\sl Koenigs function} of $f$ (the adjective ``the'' is due to the essential uniqueness coming from (3)). 

We note that, by (1), $h(\D)+i\subset h(\D)$. Moreover, $f^{\circ n}(z)=h^{-1}(h(z)+in)$ for all $z\in \D$ and $n\in\N$. 

Also, a direct computation with (4) and (5) implies that
\begin{itemize}
\item[a)] $f$ is hyperbolic if and only if $\Lambda$ is a strip,
\item[b)] $f$ is parabolic if and only if $\Lambda=\Ha, -\Ha$ or $\C$.
\end{itemize}
Parabolic maps for which
\begin{itemize}
\item[c)]  $\Lambda=\C$ are called {\sl of zero hyperbolic step} and these are exactly those parabolic maps for which $\lim_{n\to \infty}k_\D(f^{\circ n}(z), f^{\circ n}(w))=0$ for all $z,w\in \D$;
\item[d)]  $\Lambda=\Ha$ or $\Lambda=-\Ha$ are called {\sl of positive hyperbolic step} and these are exactly those for which $\lim_{n\to \infty}k_\D(f^{\circ n}(z), f^{\circ n}(w))>0$ for some---and hence any---$z,w\in \D$, $z\neq w$.
\end{itemize}

It is known (see \cite{Co} or \cite{BCDbook}) that if $f$ is hyperbolic then $\{f^{\circ n}(z)\}$ converges non-tangentially to the Denjoy-Wolff point for all $z\in \D$. It is also known that if $f$ is parabolic of positive hyperbolic step, the convergence is tangential (see \cite{Po}). In this paper we show how, in case of univalent parabolic maps with zero hyperbolic step, the type of convergence can be determined from the Euclidean shape of the image of the Koenigs map. Our argument, in fact, gives another proof of the tangential convergence in case of univalent parabolic maps with positive hyperbolic step and of the non-tangential convergence in case of univalent hyperbolic maps.

\begin{remark}\label{Rem:elliptic-model}
In case $f:\D \to \D$ is univalent and has a fixed point $z_0\in \D$, a
statement similar to Theorem~\ref{Thm:canonical-model} holds. In particular,
if  $\mu \in \C$ is chosen such that $e^{-\mu}=f'(z_0)$ (so $\Re \mu\geq 0$ by Schwarz's Lemma) one can find a univalent map $h:\D \to \C$ with $h(z_0)=0$ and
\begin{enumerate}
\item $h(f(z))=e^{-\mu} h(z)$ for all $z\in \D$,
\item $\bigcup_{m\in \N}(e^{-m\mu}h(\D))=\Lambda$, where $\Lambda=\D$ (this is the case if and only if $f$ is an automorphism of $\D$) or $\Lambda=\C$.
\item If $g:\D \to \C$ is holomorphic and $g(f(z))=e^{-\mu} g(z)$ for all $z\in \D$ then there exists a surjective, holomorphic map $\psi:\Lambda\to \bigcup_{n\geq 0}(e^{-n\mu}g(\D))$ such that $g=\psi\circ h$ and $k(\lambda z)=\lambda k(z)$ for all $z\in \Lambda$. Moreover, if $g$ is univalent, then $\psi$ is a biholomorphism. 
\end{enumerate} 
Note that $\Re\mu=0$ if and only if $f$ is an elliptic automorphism of $\D$.
\end{remark}

\section{Continuous semigroups of holomorphic self-maps of the unit disc} \label{Sec:semigroup}

\begin{definition}
A {\sl continuous semigroup of holomorphic self-maps of $\D$}, or just a {\sl semigroup in $\D$} for short, is a semigroup homeomorphism between the semigroup of real non-negative numbers (with respect to  sum), and the semigroup of holomorphic self-maps of $\D$ (with respect to composition), which is continuous when $\R^+$ is endowed with the Euclidean topology and the space of holomorphic self-maps of $\D$ is endowed with the topology of uniform convergence on compacta. \end{definition}

We refer the reader to the books \cite{Abate, BCDbook, ES, Sho} for more details about and  proofs of the following facts. 

Let $(\phi_t)$ be a semigroup in $\D$ without fixed points in $\D$. For all $t>0$, $\phi_t$ has the same Denjoy-Wolff point $\tau\in \partial \D$. In other words, $\lim_{t\to+\infty}\phi_t(z)=\tau\in \partial \D$ for all $z\in \D$. Also, there exists $\lambda\leq 0$ such that $\angle\lim_{z\to \tau}\phi_t'(z)=e^{\lambda t}$ for all $t\geq 0$.
Moreover, for all $t\geq 0$, $\phi_t$ is injective. 

In case of semigroups, the Koenigs function of each $\phi_t$ can be chosen to be independent of $t$, namely, there exists a univalent map $h:\D \to \C$ such that $h(\phi_t(z))=z+it$ for all $t\geq 0$ and $z\in \D$. The function $h$ is the Koenigs function of $\phi_1$  and all properties in Theorem~\ref{Thm:canonical-model} hold. Note that $h(\D)+it\subseteq h(\D)$ for all $t \ge 0$, in other words, in case of a semigroup, $h(\D)$ is starlike at infinity.

Similarly, in case $(\phi_t)$ has a fixed point in $\D$, the Koenigs function of each $\phi_t$ can be chosen to be independent of $t$. In this case, if $\phi_t(z_0)=z_0$ for all $t\geq 0$, then $\phi_t'(z_0)=e^{-\mu t}$ for some $\mu\in \C$, $\Re \mu\geq 0$ and for all $t\geq 0$ and
 the domain $h(\D)$ (image of the common Koenigs function of $\phi_t$) is invariant under the map $z\mapsto e^{-\mu t}z$  for all $t\geq 0$.

\section{Boundary regular fixed points}\label{Sec:BRFP}

Let $f:\D\to \D$ be holomorphic and $\sigma\in \partial \D$. If $\angle\lim_{z\to \sigma}f(z)=\sigma$, the point $\sigma$ is called a {\sl boundary fixed point of $f$}. As a consequence of the Julia-Wolff-Carath\'eodory Theorem (see, {\sl e.g.}, \cite[Prop. 1.7.4]{BCDbook} or \cite[Thm. 1.2.7]{Abate}), the non-tangential limit
\[
A:=\angle\lim_{z\to \sigma}f'(z)
\]
exists and $A\in (0,+\infty]$. Moreover, $A\leq 1$ if and only if $\sigma$ is the Denjoy-Wolff point of $f$ (or $f(z)\equiv z$). 
A boundary fixed point so that $A<+\infty$ is called a {\sl boundary regular fixed point} of $f$ with multiplier $A$. The set of all boundary regular fixed points of $f$ with multiplier $A$ is denoted by $\hbox{Fix}_A(f)$. 

In order to state the main connection between boundary regular fixed points and the canonical model for iteration, we need to introduce a notation. 

Given $\mu\in \C$, $\Re \mu> 0$, $\alpha\in (0,\pi]$ and $\theta_0\in [-\pi,\pi)$, we let
\[
\hbox{Spir}[\mu, 2\alpha, \theta_0]:=\{e^{t\mu+i\theta}: t\in \R, \theta\in (-\alpha+\theta_0,\alpha+\theta_0)\},
\]
a {\sl $\mu$-spirallike sector} of amplitude $2\alpha$. If $D\subset\C$ is a domain and $S:=\hbox{Spir}[\mu, 2\alpha, \theta_0]\subset D$, we say that $S$ is a {\sl maximal} spirallike sector in $D$ provided there exist no $\theta_1\in [-\pi,\pi)$, $\beta\in (0,\pi]$ such that $\hbox{Spir}[\mu, 2\alpha, \theta_0]\subsetneq \hbox{Spir}[\mu, 2\beta, \theta_1]\subset D$. 

We denote by $\hbox{MSpir}(\mu, \alpha, D)$ the set of all maximal $\mu$-spirallike sectors of amplitude $2\alpha$ in~$D$. 

A {\sl vertical strip} of width $R>0$ is a set of the form $\{z\in \C: a<\Re
z<a+R\}$ for some $a\in \R$. If $D\subset\C$ is a domain, a vertical strip $S$ in $D$ is {\sl maximal} provided $S\subset D$ and there is no other vertical strip contained in $D$ which properly contains $S$. 

We denote by $\hbox{MStrip}(R, D)$ the set of all maximal vertical strips of width $R$ in $D$. 

\begin{theorem}\label{BRFP-model}
Let $f:\D\to \D$ be univalent, not an automorphism of $\D$, with Koenigs function $h$ and let $A>1$.
\begin{enumerate}
\item If $f$ is elliptic, $f(z_0)=z_0$ and $f'(z_0)=e^{-\mu}$, for some $z_0
  \in \D$ and $\mu\in \C$ with $\Re\mu>0$, then there is a one-to-one
  correspondence between $\hbox{\rm Fix}_A(f)$ and $\hbox{\rm MSpir}(\mu, \frac{|\mu|^2\pi}{(\log A)(\Re \mu)}, h(\D))$.
\item If $f$ is non-elliptic then there is a one-to-one correspondence between
  $\hbox{\rm Fix}_A(f)$ and $\hbox{\rm MStrip}(\frac{\pi}{\log A}, h(\D))$.
\end{enumerate}
\end{theorem}

The proof of the previous result can be found in \cite{PP1} for the case $f$ is elliptic and $\mu\in \R$,  and in \cite{CoDi} in case $f$ is the time one map of a continuous semigroup of holomorphic self-maps of the unit disc. In the general case, the proof can be adapted from \cite[Thm. 5.6]{BCDG0} (see also \cite[Chapter 13]{BCDbook}); we leave details to the reader.

\section{Domains asymptotically starlike at infinity and their starlike-fication}\label{Sec:star}

\begin{definition}\label{Def:asy-star}
A  domain $\Omega\subsetneq \C$ is {\sl asymptotically starlike at infinity} if
\begin{enumerate}
\item $\Omega+i\subseteq \Omega$,
\item there exist $-\infty\leq a<b\leq +\infty$ such that 
\begin{equation}\label{def:base-domain}
\bigcup_{n\in \N}(\Omega-in)=(a,b)\times \R.
\end{equation}
\end{enumerate}
Moreover, we say that $\Omega$ is  asymptotically starlike at
infinity of {\sl hyperbolic type} if $a,b\in \R$, while we say that $\Omega$ is 
asymptotically starlike at infinity of {\sl parabolic type} if  $a=-\infty$ or $b=+\infty$.
\end{definition}

Note that, by definition, a domain asymptotically starlike at infinity is not required to be simply connected. If $h$ is the Koenigs function of a univalent self-map $f$ of $\D$, then the Koenigs domain  of $f$ is simply connected and  asymptotically starlike at infinity.
\begin{remark}\label{compact-in-Omega}
Let $\Omega\subset \C$ be a domain such that $\Omega+i\subset \Omega$. It is not hard to show that for every compact set $K\subset \bigcup_{n\in\N} (\Omega-in)$ there exists $N\in \N$ such that $K+iN\subset \Omega$. 
\end{remark}

Let $\Omega\subsetneq \C$ be a domain asymptotically starlike at infinity, and $z\in \Omega$. Let
\[
\tau_z:=\inf\{s\in \R: z+ir\in \Omega\ \hbox{for all} \ r>s\}.
\]

\begin{lemma}\label{Lem:tau-finite}
Let $\Omega\subsetneq \C$ be a domain asymptotically starlike at infinity. Then $\tau_z<+\infty$ for all $z\in \Omega$. 
\end{lemma}
\begin{proof}
Let $z\in \Omega$. In order to prove that $\tau_z<+\infty$, by (1) in Definition~\ref{Def:asy-star}, it is enough to prove that there exist $n\in \N$  such that $(\{w\in \C: \Re w=\Re z,  -1\leq \Im w\leq 1\}+in)\subset \Omega$. This follows at once by condition (2) in Definition~\ref{Def:asy-star} and Remark~\ref{compact-in-Omega}. 
\end{proof}

For $z\in \C$ and $r>0$, we let
\[
D(z,r):=\{w\in \C: |w-z|<r\}.
\]

\begin{lemma}\label{Lema:dist-tau}
Let $\Omega\subsetneq \C$ be a domain  asymptotically starlike at infinity, and $z \in \Omega$. Then
\begin{enumerate}
\item  $\tau_z\neq 0$. 
\item If  $\delta_\Omega(z)> 1/2$ then $\tau_z<0$.
\item If  $\tau_z>0$ there exists $p\in \C\setminus\Omega$ such that $\Re p=\Re z$ and $|z-p|\leq 1/2$.
\end{enumerate}
\end{lemma}
\begin{proof}
(1) Let $z\in \Omega$.  By (1) of Definition~\ref{Def:asy-star}, 
\[
D(z,\delta_\Omega(z))+in=D(z+in,\delta_\Omega(z))\subset\Omega
\]
 for all $n\in \N$. Since $\Omega$ is open, hence $\delta_\Omega(z)>0$, this implies that $\tau_z\neq 0$.
 
(2) If $\delta_\Omega(z)>1/2$, then $\{w\in \C: \Re w=\Re z, \Im z-1/2\leq \Im w\leq \Im z+1/2\}\subset \Omega$, hence, $z+it\in \Omega$ for all $t\geq 0$, and $\tau_z<0$.

(3) Let $r:=\inf\{|z-\Re z-it|: t\in \R, \Re z+it\not\in\Omega\}$. By (1) of Definition~\ref{Def:asy-star}, $\{\Re z+it+N, t\in (-r,r)\}\subset \Omega$ for all $N\in \N$. Hence, if $r>1/2$, $z+it\in \Omega$ for all $t \ge 0$ and $\tau_z<0$.
   \end{proof}

\begin{definition}
Let $\Omega\subsetneq \C$ be a domain asymptotically starlike at infinity. The {\sl starlike-fication} of $\Omega$ is the set $\Omega^\ast$ defined by
\[
\Omega^\ast:=\{z\in \Omega: \tau_z<0\}.
\] 
\end{definition}

\begin{lemma}\label{Lem:starlike-fication}
Let $\Omega\subsetneq \C$ be a domain asymptotically starlike at
infinity and $\Omega^\ast$ its starlike-fication. Then $\Omega^\ast\not=\emptyset$ is a  simply connected domain starlike at infinity, $\Omega^\ast\subseteq \Omega$ and
\begin{equation}\label{Eq:same-base}
\bigcup_{n\in \N}(\Omega-in)=\bigcup_{t\geq 0}(\Omega^\ast-it)
\end{equation}
\end{lemma}
\begin{proof}
If $z\in \Omega$, by Lemma~\ref{Lem:tau-finite}, $\tau_z<+\infty$. Hence, for every $t>\tau_z$, $z+it\in \Omega$, which implies that $\tau_{z+it}<0$ for all $t>\tau_z$, that is, $z+it\in \Omega^\ast$ for all $t>\tau_z$, proving that $\Omega^\ast$ is  non-empty.

If $z\in \Omega^\ast$, then by definition of $\tau_z$, $z+it\in \Omega$ for all $t\geq 0$. In particular, $\tau_{z+it}<0$ for all $t\geq 0$, hence $z+it\in \Omega^\ast$ for all $t\geq 0$. Thus $\Omega^*$ is starlike at infinity.
 
We show that $\Omega^\ast$ is open. For $w_0\in \C$ and $a,b>0$ let 
\[
D(w_0,a,b):=\{z\in \C: |\Re z-\Re w_0|<a, |\Im z-\Im w_0|<b\}.
\]

Let $z_0\in \Omega^\ast$. Assume by contradiction that there exists a sequence $\{z_n\}\subset\C\setminus\Omega^\ast$ such that $\lim_{n\to \infty}z_n=z_0$. 

Since, in particular, $z_0\in \Omega$,  there exist $0<\epsilon_1,\epsilon_2<1$ such that 
\begin{equation}\label{Eq1-pf-starlike}
\overline{D(z_0,\epsilon_1,\epsilon_2)}\subset \Omega.
\end{equation}
Without loss of generality, we can assume that $\{z_n\}\subset D(z_0,\epsilon_1,\epsilon_2)$, which means that $z_n\in \Omega\setminus\Omega^\ast$ for every $n\in \N$. In particular, for every $n\in \N$ there exists $t_n>0$ such that $z_n+it_n\not\in \Omega$. 

By Remark~\ref{compact-in-Omega} and \eqref{Eq1-pf-starlike}, there exists $N_0\in \N$ such that $D(z_0,\epsilon_1,2)+iN_0 \subset \Omega$. Thus,  by (1) in Definition~\ref{Def:asy-star}, 
\[
D(z_0+it,\epsilon_1,2)=D(z_0,\epsilon_1,2)+it\subset  \Omega
\]
for all $t\geq N_0$. Since $z_n+it_n\in D(z_0+it_n, \epsilon_1, \epsilon_2)\subset D(z_0+it_n, \epsilon_1, 2)$, it follows  that $\sup_{n\in \N} t_n<+\infty$. Thus, up to extracting subsequences, we can  assume that $\{t_n\}$ converges to some $t_0\geq 0$. Therefore, $\lim_{n\to \infty}(z_n+it_n)=z_0+it_0$. Since $z_n+it_n\in \C\setminus\Omega$, which is closed, it follows that $z_0+it_0\not\in\Omega$. But then, by definition of $\Omega^\ast$, we have $z_0\not\in\Omega^\ast$, a contradiction.

Next, it is easy to see that $\Omega^\ast$ is a simply connected domain. In fact, by a similar argument which we have used to show that $\Omega^*$ is open, one can prove that for any $z, w \in \Omega^*$ there is $L \in \N$ such that the line segment joining $z+iL$ and $w+iL$ is contained in $\Omega^*$. This immediately yields the pathconnectedness of $\Omega^*$. If $\Gamma$ is a closed curve in $\Omega^*$ and $z \in \C\setminus \Gamma$ such that $\Gamma$ has nonvanishing winding number around $z$, then  there is  a point $p\in \Gamma\subset \Omega^*$ with $\Re p=\Re z$ and $\Im p<\Im z$, so $z \in \Omega^*$. Hence $\Omega^*$ is simply connected.

Finally, \eqref{Eq:same-base} follows at once because, by Remark~\ref{compact-in-Omega}, for every $z\in \Omega$ there exists $t_0 \ge 0$ such that $z+it\in \Omega$ for all $t\geq t_0$. 
\end{proof}

Let $\Omega\subsetneq \C$ be a domain which is parabolic asymptotically
starlike at infinity.
We aim to localize the hyperbolic metric of $\Omega$ with respect to that of
$\Omega^\ast$, and start with a definition:

\begin{definition}
If $D\subset \C$ is a domain and $r>0$, we let 
\[
D_r:=\{z\in D: \delta_D(z)>r\}.
\]
\end{definition}

\begin{lemma}\label{Lem:para-non-empty}
Let $\Omega\subsetneq \C$ be a domain asymptotically starlike at infinity of parabolic type. Then $\Omega_r\neq\emptyset$ for all $r>0$.
\end{lemma}
\begin{proof}
Let $a,b$ as in (2) of Definition~\ref{Def:asy-star}. Since $\Omega$ is parabolic, we have $a=-\infty$ or $b=+\infty$. This implies that for any $r>0$ there exists $x\in \R$  such that  $\overline{D(x,r)}\subset (a,b)\times \R$. By Remark~\ref{compact-in-Omega}, there exists $n\in \N$ such that $\overline{D(x,r)}+in\subset \Omega$. Therefore, $\delta_\Omega(x+in)>r$ and $\Omega_r$ is non-empty.
\end{proof}

\begin{remark}\label{Rem:inside-Omegastar}
Let $\Omega\subsetneq \C$ be a domain asymptotically starlike
at infinity of parabolic type. Then, by Lemma~\ref{Lema:dist-tau},   $\Omega_r\subset
\Omega^\ast$ for all $r\geq 1/2$. Hence the starlike-fication $\Omega^*$ is in
a sense a `large' subset of $\Omega$.
\end{remark}

\begin{proposition}\label{Prop:loc-1}
Let $\Omega\subsetneq \C$ be a domain starlike at infinity of parabolic type. Then for every $R>0$ there exists $r>1$ such that for every $z\in \Omega_r$, 
\[
B_\Omega(z, R)\subset \Omega_{1}\subset \Omega^\ast.
\]
\end{proposition}
\begin{proof}
Assume by contradiction that there exists $R>0$ such that for every $n\in \N$ there exist $z_n\in \Omega_n$ and $q_n\in \Omega$,  such that $\delta_\Omega(q_n)\leq 1$ and $k_{\Omega}(z_n,q_n)<R$. 

Let $\tilde q_n\in \partial \Omega$ be such that $|q_n-\tilde q_n|\leq 1$. By (1) of Definition~\ref{Def:asy-star}, $\tilde q_n-i\not\in\Omega$.  Let $T_n:\C \to \C$ be the translation defined by $T_n(z)=z-\tilde q_n+i$. Note that $T_n(\tilde q_n)=i$, $T_n(\tilde q_n-i)=0$ and 
\[
|T_n(q_n)-i|=|q_n-\tilde q_n|\leq 1.
\]
Moreover, let $V:=\C\setminus\{0,i\}$. Taking into account that $\Omega\subset \C\setminus\{\tilde q_n, \tilde q_n-i\}$, we have
\[
\delta_V(T_n(z_n))=\delta_{\C\setminus\{\tilde q_n, \tilde q_n-i\}}(z_n)\geq \delta_\Omega(z_n)>n.
\]
Finally, note that $\Omega$ is biholomorphic to $T_n(\Omega)$ via $T_n$. Therefore, since $T_n(\Omega)\subset V$,
\[
R>k_\Omega(q_n,z_n)=k_{T_n(\Omega)}(T_n(q_n), T_n(z_n))\geq k_V(T_n(q_n), T_n(z_n)).
\]
Let $A:=\partial D(i, 1)$ and $B_n:=\partial(D(0,n)\cup D(i,n))$. Let 
\[
k_n:=k_V(A,B_n):=\inf\{k_V(z,w): z\in A\setminus \{0\}, w\in B_n\}.
\]
 Since $V$ is complete hyperbolic, $\lim_{n\to \infty}k_n=+\infty$ (otherwise we would find a sequence in $V$ converging to infinity which stays at finite hyperbolic distance from a compact subset of $V$).

By the previous considerations, $T_n(q_n)\in \overline{D(i, 1)}$, while $T_n(z_n)\in \C\setminus (\overline{D(0,n)}\cup\overline{D(i,n)})$. Therefore, 
\[
  k_n=k_V(A,B_n)
 \le k_V(T_n(q_n),T_n(z_n)) <R,
\]
a contradiction.
\end{proof}

The next result is a sort of converse of the previous one:
\begin{proposition}\label{Prop:loc-2}
Let $\Omega\subsetneq \C$ be a domain asymptotically starlike at infinity of parabolic type. Then there exists $S>0$ such that  for every $z\in \Omega_1$, 
\[
B_\Omega(z, S)\subset  \Omega^\ast.
\]
\end{proposition}
\begin{proof}
The argument is similar to the one used in Proposition~\ref{Prop:loc-1}, so we just sketch the proof.

Assume by contradiction that for every $n\in \N$ there exist $z_n\in \Omega_1$ and $q_n\in \Omega\setminus\Omega^\ast$ such that $k_\Omega(z_n, q_n)<\frac{1}{n}$. By Lemma~\ref{Lema:dist-tau}, $\delta_\Omega(q_n)\leq 1/2$. Therefore, using the translation $T_n$ as in the proof of Proposition~\ref{Prop:loc-1}, and keeping the same notation, we have
\[
\frac{1}{n}>k_\Omega(q_n,z_n)=k_{T_n(\Omega)}(T_n(q_n), T_n(z_n))\geq k_V(T_n(q_n), T_n(z_n)).
\]
Since $\delta_V(T_n(z_n))>1$ and $\delta_V(T_n(q_n))\leq 1/2$ (see again the proof of Proposition~\ref{Prop:loc-1}), we obtain a contradiction for $n\to \infty$.
\end{proof}

Now we are ready to show that in $\Omega_1$, the (infinitesimal) hyperbolic metrics of $\Omega$ and $\Omega^\ast$ are equivalent:

\begin{theorem}\label{Thm:loc-inf}
Let $\Omega\subsetneq \C$ be a domain asymptotically starlike at infinity of parabolic type. Then there exists $c>1$ such that for every $z\in \Omega_1$ and $v\in \C$,
\[
\kappa_\Omega(z;v)\leq \kappa_{\Omega^\ast}(z;v)\leq c \kappa_\Omega(z;v).
\]
\end{theorem}
\begin{proof}
Let $S>0$ be given by Proposition~\ref{Prop:loc-2}. Hence, for every $z\in \Omega_1$, $B_\Omega(z, S)\subset  \Omega^\ast\subset \Omega$, from which we get for all $v\in \C$
\[
\kappa_\Omega(z;v)\leq \kappa_{\Omega^\ast}(z;v)\leq  \kappa_{B_\Omega(z, S)}(z;v).
\]
The result then follows from Lemma~\ref{lem:local-disc}.
\end{proof}

As an immediate consequence of Theorem~\ref{Thm:loc-inf}, we have:
\begin{corollary}\label{Cor:estimk}
Let $\Omega\subsetneq \C$ be a domain asymptotically starlike at infinity of parabolic type. Then for every absolutely continuous curve $\gamma:[0,1]\to \Omega_1$,
\[
\ell_\Omega(\gamma)\leq \ell_{\Omega^\ast}(\gamma)\leq c \ell_\Omega(\gamma),
\]
where $c>1$ is given by Theorem~\ref{Thm:loc-inf}.
In particular, if $z,w\in \Omega_1$ are two points such that a  geodesic in  $\Omega$ joining $z$ with $w$ is contained in $\Omega_1$, then
\[
k_\Omega(z,w)\leq k_{\Omega^\ast}(z,w)\leq c k_\Omega(z,w).
\]
\end{corollary}

\section{Gromov hyperbolic domains  asymptotically starlike at infinity of parabolic type}\label{Sec:Gromov}

In this section we assume that $\Omega\subsetneq \C$ is a Gromov hyperbolic
domain  asymptotically starlike at infinity of parabolic type.
In particular, our discussion includes the case when $\Omega$  is the Koenigs domain of a parabolic univalent self-map of the unit disc.

\begin{theorem}\label{Thm:sigma-qg_gen}
  Let $\Omega\subsetneq \C$ be a Gromov hyperbolic domain asymptotically starlike at infinity of parabolic type. Assume that $\omega : [1,+\infty) \to \Omega^*$ is an absolutely continuous curve such that
  $$ \lim \limits_{t \to +\infty} \delta_{\Omega}(\omega(t))=+\infty \, . $$
  Then $\omega$ is a quasi-geodesic in $\Omega$ if and only if $\omega$ is a quasi-geodesic in $\Omega^*$.
\end{theorem}

\begin{proof} The proof of the only-if part is straightforward. In fact, by assumption, there is $t_1 \ge 0$ such that $\omega(t) \in \Omega_1$ for all $t \ge t_1$. If $c>1$ denotes the constant from Theorem \ref{Thm:loc-inf} and if $\omega : [1,+\infty) \to \Omega^*$ is a $(A,B)$-quasi-geodesic in $\Omega$, then  it follows directly from the definitions and Corollary \ref{Cor:estimk} that $\omega : [t_1,+\infty)$ is a $(Ac,Bc)$-quasi-geodesic in $\Omega^*$, so $\omega : [1,+\infty)$ is a $(Ac,B')$-quasi-geodesic in $\Omega^*$  with $B'= (1+Ac)\ell_{\Omega^*}(\omega;[1,t_1])+Bc$.
We next prove the if-part, which is more difficult to handle. Let $\omega : [1,+\infty) \to \Omega^*$ be a $(A,B)$-quasi-geodesic in $\Omega^*$, that is,
 \begin{equation}\label{Eq:AB-qg}
\ell_{\Omega^\ast}(\omega;[s,t])\leq A k_{\Omega^\ast}(\omega(s),\omega(t))+B
\end{equation} 
  for all  $s \le t$ in $[1,+\infty)$.
Let $G>0$ be the Gromov constant of $\Omega$ and let $c>1$ be given by Theorem~\ref{Thm:loc-inf}.
By $M>0$ we denote  the constant given by the Gromov shadowing lemma for the $(Ac,B)$-quasi-geodesics of $\Omega$ (see \eqref{Eq:shadows}).

Finally, let $R'\geq G+M$, $R=2R'$ and let $r>1$ be given by Proposition~\ref{Prop:loc-1}. Since
$\delta_{\Omega}(\omega(t)) \to +\infty$ as $t \to +\infty$, 
 there exists $t_r>0$ such that $\omega(t)\in \Omega_r$ for every $t\geq t_r$.

By \eqref{Eq:AB-qg}, taking into account that $\Omega^\ast\subset\Omega$, we
have for all $s \le t$ in $[1,+\infty)$,
\begin{equation}\label{eq:est1}
\ell_\Omega(\omega;[s,t])\leq \ell_{\Omega^\ast}(\omega;[s,t])\leq A k_{\Omega^\ast}(\omega(s),\omega(t))+B.
\end{equation}
Now, take $ s_0\geq t_r$ and let $t_0>s_0$ be such that $\omega(t)\in B_\Omega(\omega(s_0), R')$ for all $t\in [s_0,t_0]$. By Proposition~\ref{Prop:loc-1}, $B_\Omega(\omega(s_0), R)\subset \Omega_1$. By Lemma~\ref{Lem:disc-quasi-total},  since for every $z,w\in B_\Omega(\omega(s_0), R')$ the geodesic  in $\Omega$ joining $z$ to $w$ is contained in $B_\Omega(\omega(s_0), R)$, it follows from Corollary~\ref{Cor:estimk} that 
\[
k_{\Omega^\ast}(\omega(s), \omega(t))\leq c k_{\Omega}(\omega(s), \omega(t))
\]
 for all $s,t\in [s_0,t_0]$. Therefore, by \eqref{eq:est1},
 $\omega|_{[s_0,t_0]}$ is an $(Ac, B)$-quasi-geodesic in $\Omega$ for every $t_r\leq s_0<t_0$ such that $\omega(t)\in B_\Omega(\omega(s_0), R')$ for all $t\in [s_0,t_0]$.
 
 Now,  we  prove the following statement for all $N\in \N$, $N\geq 1$: 
\medskip 

\noindent{\sl $(A_N)$} if  $t_r\leq a< b$ and there exist
$a=s_0<s_1<\ldots<s_N=b$,  such that $\omega(t)\in B_\Omega(\omega(s_j), R')$
for all $t\in [s_j,s_{j+1}]$, $j=0,\ldots, N-1$ then $\omega|_{[a,b]}$ is a
$(Ac, B)$-quasi-geodesic in $\Omega$. 

\medskip
 
We argue by induction. We already proved that $(A_1)$ holds.  Assuming that $(A_j)$ holds for $j=1,\ldots, N$, we have to prove that $(A_{N+1})$ holds as well.

Let then $t_r\leq a< b$ and assume there exist $a=s_0<s_1<\ldots<s_{N+1}=b$ such that $\omega(t)\in B_\Omega(\omega(s_j), R')$ for all $t\in [s_j,s_{j+1}]$, $j=0,\ldots, N$. We have to show that for every $a\leq s\leq t\leq b$, 
\[
\ell_\Omega(\omega;[s,t])\leq Ac k_\Omega(\omega(s), \omega(t))+B.
\]
By induction, if $t\leq s_N$, or $s\geq s_N$,  the result is true. So we can
assume that $s\in [s_0,s_N)$ and $t\in (s_N, s_{N+1}]$. By \eqref{eq:est1},
and arguing as before, it is enough to prove that a geodesic $\gamma:[0,1]\to
\Omega$ in $\Omega$ such that $\gamma(0)=\omega(s)$ and $\gamma(1)=\omega(t)$,
is contained in $\Omega_1$.  To this aim, let $\gamma_1:[0,1]\to \Omega$ be a
geodesic in $\Omega$ such that $\gamma_1(0)=\omega(s)$ and
$\gamma_1(1)=\omega(s_N)$ and let $\gamma_2:[0,1]\to \Omega$ be a geodesic in
$\Omega$ such that $\gamma_2(0)=\omega(s_N)$ and
$\gamma_2(1)=\omega(t)$. Since $\omega|_{[s,s_N]}$ and $\omega|_{[s_N,t]}$ are
$(Ac, B)$-quasi-geodesics in $\Omega$ by induction hypothesis, it follows from Gromov's shadowing lemma that for each $u\in [0,1]$, there exist $u_1\in [s,s_N]$ and $u_2\in [s_N, t]$ such that 
\begin{equation}\label{Eq:G-qc1}
k_\Omega(\gamma_j(u), \omega(u_j))<M, \quad j=1,2.
\end{equation}

Let now $w\in [0,1]$. Since $\Omega$ is Gromov hyperbolic, and $\{\gamma, \gamma_1,\gamma_2\}$ is a geodesic triangle in $\Omega$, $\gamma(w)$ stays at  hyperbolic distance less than $G$ from $\gamma_1\cup\gamma_2$. Thus,   there exists $u\in [0,1]$ such that 
\[
\min_{j=1,2} k_\Omega(\gamma(w),\gamma_j(u))<G.
\]
We can assume that $k_\Omega(\gamma(w),\gamma_1(u))<G$ (the case $k_\Omega(\gamma(w),\gamma_2(u))<G$ is similar).

Therefore, if $u_1$ is the point given by \eqref{Eq:G-qc1}, we have
\[
k_\Omega(\gamma(w),\omega(u_1))\leq k_\Omega(\gamma(w),\gamma_1(u))+k_\Omega(\gamma_1(u),\omega(u_1))<G+M.
\]
Since $R> G+M$ and $\omega(u_1)\in \Omega_r$, it follows from Proposition~\ref{Prop:loc-1} that $\gamma(w)\in \Omega_1$, and hence, by the arbitrariness of $w\in[0,1]$, $(A_{N+1})$ holds.

Now, since by compactness for every $t_r\leq a\leq b$ one can cover
$\omega([a,b])$ with a finite number of hyperbolic balls of radius $R'$
centered at points of $\omega([a,b])$, it follows from $(A_N)$ that
$\omega|_{[t_r,+\infty)}$ is a $(Ac,B)$-quasi-geodesic in $\Omega$. Therefore,
taking $B':= (1+Ac) \ell_\Omega(\omega;[1,t_r])+B$, we see that $\omega:[1,+\infty)\to \Omega$ is a $(Ac,B')$-quasi-geodesic in $\Omega$.
\end{proof}

We are now in a position to construct from the euclidean shape of $\Omega$ a curve $\sigma : [1,+\infty) \to \Omega$, which is 
a quasi-geodesic both in $\Omega$ and $\Omega^*$.

\medskip
{\sl Assumption:} We assume that $0\not\in\Omega$ and $it\in \Omega$ for all $t>0$.
\medskip

Note that, since $\Omega$ is asymptotically starlike at infinity of parabolic type and $\Omega\neq\C$, unless $\Omega$ is a vertical half-plane (and hence it is simply connected and starlike at infinity), there always exists $p\in \C$ such that $p\not\in\Omega$ and $p+it\in \Omega$ for all $t>0$, so  $\Omega-p$ satisfies the previous Assumption.

Note also that if $\Omega$ satisfies the Assumption, then $0\not\in\Omega^\ast$ and $it\in \Omega$ for all $t>0$.
 We define  $\sigma:[1,+\infty)\to \Omega^\ast$ by
\begin{equation} \label{eq:sigma}
\sigma(t):=\frac{\delta^+_{\Omega^\ast,0}(t)-\delta^-_{\Omega^\ast,0}(t)}{2}+it.
\end{equation}
In \cite[Lemma 4.1]{BCDGZ}, it is proved that $\sigma$ is $2$-Lipschitz and in \cite[Theorem 4.2]{BCDGZ} it is shown that $\sigma$ is a quasi-geodesic in $\Omega^\ast$. 
Our aim is to show that $\sigma$ is a quasi-geodesic in $\Omega$. 
\begin{lemma}\label{Lem:curve-in}
  Let $\Omega \subsetneq \C$ be a domain asymptotically starlike at infinity of parabolic type 
  such that $0 \not \in \Omega$ and $it \in \Omega$ for all $t>0$. Then for every $t>0$, 
\[
\delta^\pm_{\Omega^\ast,0}(t)\leq \delta^\pm_{\Omega, 0}(t)\leq \delta^\pm_{\Omega^\ast,0}(t)+\frac{1}{2}.
\]
In particular, $\lim \limits_{t \to +\infty} \delta_{\Omega}(\sigma(t))=+\infty$.
\end{lemma}
\begin{proof}
Since $\Omega^\ast\subset\Omega$, it is clear that $\delta^\pm_{\Omega^\ast,0}(t)\leq \delta^\pm_{\Omega, 0}(t)$ for every $t>0$. In order to prove the other inequality,  let $t>0$ and let $p\in \C\setminus\Omega^\ast$, $\Re p\geq 0$, be such that $\delta^+_{\Omega^\ast,0}(t)=|t-p|$. If $p\not\in\Omega$, then $\delta^+_{\Omega^\ast,0}(t)= \delta^+_{\Omega, 0}(t)$. In case $p\in \Omega\setminus\Omega^\ast$, then $\tau_p>0$ and, by Lemma~\ref{Lema:dist-tau} (3), there exists $q\in \C\setminus\Omega$, $\Re q=\Re p$ such that $|p-q|\leq 1/2$. Therefore,
\[
\delta^+_{\Omega^\ast, 0}(t)=|it-p|\geq |it-q|-|p-q|\geq \delta^+_{\Omega,0}(t)-\frac{1}{2}.
\] 
A similar argument leads to the corresponding inequality for $\delta^-_{\Omega,0}(t)$.

Finally, a simple geometric argument as in \cite[Lemma 4.3]{BCDGZ} and using  $\delta^\pm_{\Omega^\ast,0}(t)\leq \delta^\pm_{\Omega, 0}(t)$ shows 
\begin{equation} \label{eq:dist}
\delta_{\Omega}(\sigma(t)) \ge \frac{1}{2 \sqrt{2}} \left( \delta^+_{\Omega^*,0}(t)+\delta^-_{\Omega^*,0}(t) \right) \, , \qquad t \ge 1 \, .
\end{equation}
Since $\Omega$ is parabolic, Lemma \ref{Lem:starlike-fication} implies that
$$ \bigcup \limits_{t \ge 0} (\Omega^*-it)$$
contains a vertical half-plane, so 
$\delta^+_{\Omega^\ast,0}(t)\to +\infty$ or $\delta^-_{\Omega^\ast,0}(t)\to +\infty$ as $t\to+\infty$.
\end{proof}

\begin{theorem}\label{Thm:sigma-qg}
Let $\Omega\subsetneq \C$ be a Gromov hyperbolic domain asymptotically starlike at infinity of parabolic type. Assume that $0\not\in\Omega$ and $it\in \Omega$ for all $t>0$. Then  $\sigma:[1,+\infty)\to \Omega$ defined by (\ref{eq:sigma}) is a quasi-geodesic in $\Omega$.
\end{theorem}

\begin{proof}
  By Lemma \ref{Lem:curve-in}, this follows from  Theorem \ref{Thm:sigma-qg_gen} and \cite[Theorem 4.2]{BCDGZ}.
\end{proof}

\section{Simply connected domains asymptotically starlike at infinity of parabolic type}

Let $\Omega\subsetneq \C$ be a simply connected domain  asymptotically starlike at infinity of parabolic type. We assume also that $0\not\in\Omega$ and $it\in \Omega$ for all $t>0$.

Let $h:\D\to \Omega$ be a Riemann map. Up to precomposing $h$ with a rotation we can assume that $\lim_{t\to+\infty}h^{-1}(it)=1$.

Let $\Omega^\ast$ be the starlike-fication of $\Omega$ and let $h^\ast:\D\to
\Omega^\ast$ be a Riemann map such that
$\lim_{t\to+\infty}(h^\ast)^{-1}(it)=1$. Finally, let $\sigma$ be the curve
defined by (\ref{eq:sigma}).

\begin{proposition}\label{Prop:same-in-d}
Let $\{z_n\}\subset\D$ be a sequence converging to $1$. Then the following are equivalent:
\begin{enumerate}
\item $\{z_n\}$ converges non-tangentially to $1$.
\item There exists $C_1>0$ such that $k_\Omega(\sigma([1,+\infty)), h(z_n))<C_1$ for all $n\in \N$.
\item There exists $C_2>0$ such that $k_{\Omega^\ast}(\sigma([1,+\infty)), h^\ast(z_n))<C_2$ for all $n\in \N$.
\end{enumerate}
\end{proposition}
\begin{proof}
By \cite[Lemma 5.2]{BCDGZ}, $\lim_{t\to+\infty}(h^\ast)^{-1}(\sigma(t))=1$. The same argument used in such a lemma can be applied to show that $\lim_{t\to+\infty}h^{-1}(\sigma(t))=1$ as well---very sketchy, if this were not the case, the horizontal segments joining $\sigma(n)$ to $in$, $n\in\N$ would form a sequence of Koebe's arcs for $h$, contradicting the no Koebe's arcs Theorem. 

Now,  $\Omega$ is simply connected, hence $(\Omega, k_\Omega)$ is Gromov hyperbolic. By Theorem~\ref{Thm:sigma-qg}, $\sigma$ is a quasi-geodesic both in $\Omega$ and $\Omega^\ast$, and hence the result follows from Proposition~\ref{non-tg-sector-hyp}.
\end{proof}

\begin{theorem}\label{Thm:same-conv}
Let $\{w_n\}\subset \Omega$. Then the following are equivalent:
\begin{enumerate}
\item  $\{h^{-1}(w_n)\}$ converges non-tangentially to $1$.
\item $\{w_n\}$ is eventually contained in $\Omega^{\ast}$ and
  $\{(h^\ast)^{-1}(w_n)\}$ converges non-tangentially to $1$.
\end{enumerate}
\end{theorem}
\begin{proof}
If (1) holds, then by Proposition~\ref{Prop:same-in-d}, there exists $C>0$ such that $k_{\Omega}(\sigma([1,+\infty),w_n)< C$ for all $n\in \N$. Thus, for each $n\in \N$, we can find  $s_n\in [1,+\infty)$  such that
\[
k_\Omega(\sigma(s_n),w_n)< C.
\]
On the other hand, since for every $T>1$ the set $\sigma([1,T])$ is compact and $\{w_n\}$ eventually exits all compacta of $\Omega$, it follows that
\[
\lim_{n\to +\infty}k_\Omega(\sigma([1,T]),w_n)=+\infty.
\]
Therefore, $\{s_n\}$ eventually leaves each compact interval $[1,T]$, and by Lemma~\ref{Lem:curve-in}, for every $r>0$ there exists $n_r\in \N$ such that $\sigma(s_n)\in \Omega_r$ for all $n\geq n_r$. 

Let $R:=C$ and let $r>0$ be given by Proposition~\ref{Prop:loc-1}. Then, for each $n\geq n_r$, $\sigma(s_n)\in \Omega_r$ and, by Proposition~\ref{Prop:loc-1}, $B_\Omega(\sigma(s_n), C)\subset \Omega_1\subset\Omega^\ast$. 

In particular, $w_n\in \Omega^\ast$ for all $n\geq n_r$. Moreover, since $w_n\in B_\Omega(\sigma(s_n), C)$ and $B_\Omega(\sigma(s_n), C)$ is totally geodesic in $\Omega$ (see Remark \ref{disc-total-geo}), the geodesic in $\Omega$ joining $\sigma(s_n)$ to $w_n$ is contained in $\Omega_1$. It therefore follows from Corollary~\ref{Cor:estimk} that for all $n\geq n_r$,
\[
k_{\Omega^\ast}(\sigma(s_n),w_n)\leq c k_{\Omega}(\sigma(s_n),w_n)< cC.
\]
Therefore, by Proposition~\ref{Prop:same-in-d}, $\{(h^\ast)^{-1}(w_n)\}$ converges non-tangentially to $1$, and (2) holds.

If (2) holds, then by Proposition~\ref{Prop:same-in-d}, there exists $C>0$ such that $k_{\Omega^\ast}(\sigma([1,+\infty)),w_n)\leq C$ for all $n\geq n_0$. Therefore, 
\[
k_{\Omega}(\sigma([1,+\infty)),w_n)\leq k_{\Omega^\ast}(\sigma([1,+\infty)),w_n)< C,
\]
and, again by Proposition~\ref{Prop:same-in-d}, $\{h^{-1}(w_n)\}$ converges non-tangentially to $1$. 
\end{proof}

\section{Semigroup-fication and the proofs of Theorem~\ref{Thm:semigroupfication-intro} and Theorem~\ref{Thm:main}}\label{Sec:proof}

Let $f:\D \to \D$ be univalent with no fixed points in $\D$. Up to conjugation with a rotation, we can assume that $1$ is the Denjoy-Wolff point of $f$. Let $h:\D \to \C$ be the Koenigs function of $f$ given by Theorem~\ref{Thm:canonical-model}  and $\Omega:=h(\D)$ the Koenigs domain of $f$. 

Note that $\Omega$ is simply connected and asymptotically starlike at infinity. Moreover, $\Omega$ is parabolic if and only if $f$ is parabolic. Let $\Omega^\ast\subset \Omega$ be the starlike-fication of $\Omega$. By Lemma~\ref{Lem:starlike-fication}, $\Omega^\ast$ is a non-empty simply connected domain which is  starlike at infinity. Let $h^\ast:\D \to \Omega^\ast$ be a Riemann map and, for $t\geq 0$ and $z\in \D$, let
\[
\phi_t(z):=(h^\ast)^{-1}(h^\ast(z)+it).
\]
It is easy to see that $(\phi_t)$ is a continuous semigroup of holomorphic self-maps in $\D$ and $h^\ast$ is the Koenigs function of $(\phi_t)$. In view of \eqref{Eq:same-base}, $f$ is hyperbolic ({\sl respectively}, parabolic of positive/zero hyperbolic step) if and only if $(\phi_t)$ is hyperbolic ({\sl resp.}, parabolic of positive/zero hyperbolic step). Up to conjugation with a rotation, we can assume  that $1$ is the Denjoy-Wolff point of $(\phi_t)$.

\begin{remark}
It follows from Theorem~\ref{Thm:canonical-model}(4) and \eqref{Eq:same-base}, that if $f$ is hyperbolic (namely, $\alpha=\angle\lim_{z\to 1}f'(z)\in (0,1)$), then, setting $\lambda:=\log \alpha$,  
\[
\angle\lim_{z\to 1}\phi_t'(z)=e^{\lambda t}. 
\]
In other words, $f$ and $\phi_1$ have the same dilation coefficient at the Denjoy-Wolff point.
\end{remark}

In case $f$ fixes a point $z_0\in \D$, $f'(z_0)=e^{-\mu}$ for some $\mu\in \C$, $\Re \mu\geq 0$, and $h$ is the Koenigs function of $f$ (see Remark~\ref{Rem:elliptic-model}) we denote $\Omega:=h(\D)$. In this case, we can define a $\mu$-spirallike domain $\Omega^\ast$ starting from $\Omega$ by declaring $z\in \Omega^\ast$ if and only if $e^{-\mu t}z\in \Omega$ for all $t\geq 0$. Arguing in a similar fashion as we did before, one can easily prove that $\Omega^\ast$ is a simply connected domain, different from $\C$, $\mu$-spirallike ({\sl i.e.}, $e^{-t\mu}\Omega^\ast\subseteq \Omega^\ast$ for all $t\geq 0$) and contains $0$. Let $h^\ast:\D \to \Omega^\ast$ be a Riemann map and, for $t\geq 0$ and $z\in \D$, let
\[
\phi_t(z):=(h^\ast)^{-1}(e^{-\mu t}h^\ast(z)).
\]
As before, it  is easy to see that $(\phi_t)$ is a continuous semigroup of holomorphic self-maps in $\D$, $\phi_t(z_0)=z_0$ and $h^\ast$ is the Koenigs function of $(\phi_t)$. Moreover, $f'(z_0)=e^{-\mu}=\phi_1'(z_0)$.

\begin{definition}
The semigroup $(\phi_t)$ is called the {\sl semigroup-fication} of $f$.
\end{definition}

\begin{lemma}\label{Lem:premain}
Let $f:\D \to \D$ be univalent without fixed points in $\D$ and let $(\phi_t)$ be the semigroup-fication of $f$. Suppose that  $\{n_k\}\subset\N$ is an increasing sequence converging to $\infty$. Then the following are equivalent:
\begin{enumerate}
\item $\{f^{\circ n_k}(z)\}$ converges non-tangentially to the Denjoy-Wolff point of $f$ for some---and hence any---$z\in \D$,
\item $\{\phi_{n_k}(z)\}$ converges non-tangentially to the Denjoy-Wolff point of $(\phi_t)$  for some---and hence any---$z\in \D$.
\end{enumerate}
\end{lemma}
\begin{proof}
As remarked before, we can assume that $1$ is the Denjoy-Wolff point of both $f$ and $(\phi_t)$.

(i) We first  note that for all $z,w\in \D$ and $k\in \N$
\[
k_\D(f^{\circ n_k}(z), f^{\circ n_k}(w))\leq k_\D(z,w).
\]
Therefore, the sequences $\{f^{\circ n_k}(z)\}$ and $\{f^{\circ n_k}(w)\}$ stay at finite hyperbolic distance. By \cite[Proposition 4.5]{BCDG} (see also \cite[Corollary 6.2.6]{BCDbook}), a sequence $\{z_n\}\subset\D$ converges non-tangentially to $1$ if and only if it stays at finite hyperbolic distance from $[0,1)$. Therefore, by the triangle inequality, it follows that $\{f^{\circ n_k}(z)\}$ converges to $1$ non-tangentially if and only if so does $\{f^{\circ n_k}(w)\}$. A similar argument shows that $\{\phi_{ n_k}(z)\}$ converges to $1$ non-tangentially if and only if so does $\{\phi_{ n_k}(w)\}$ for any $w\in \D$. 

(ii) We now prove Lemma \ref{Lem:premain} in case  $f$ is parabolic, that is,  $\Omega$ is a simply connected domain which is asymptotically starlike at infinity of parabolic type.
If $\Omega=\Omega^\ast$, there is nothing to prove. Otherwise
there exists $p\in \C\setminus\Omega$ such that $p+it\in \Omega$ for all $t>0$. Up to  a translation, we can assume $p=0$.  Let $w_{k}:=i+n_ki$. Note that $\{w_k\}\subset\Omega^\ast$ for all $k\in \N$. By Theorem~\ref{Thm:same-conv}, $\{h^{-1}(w_k)\}$ converges non-tangentially to $1$ if and only if so does  $\{(h^\ast)^{-1}(w_k)\}$. Since 
\[
h^{-1}(w_k)=f^{\circ n_k}(h^{-1}(i)), \quad (h^\ast)^{-1}(w_k)=\phi_{n_k}((h^\ast)^{-1}(i)),
\]
the equivalence of (1) and (2)  follows from part (i).

(iii) Finally, we assume that $f$ is hyperbolic, so that $(\phi_t)$ is hyperbolic. As we already remarked, it is known that in this case both $\{f^{\circ n}(z)\}$ and $\{\phi_{n}(z)\}$ converge
non-tangentially to $1$, so that ``(1) $\Leftrightarrow$ (2)'' is trivially true. However, there is also a way to prove this last statement along the same lines as in the parabolic case above. We sketch it here.

Let $\Lambda:=\{0<\Re w<a\}=\bigcup_{n\in \N} (\Omega-in)$, $a>0$. Up to translation, we can assume that $\frac{a}{2}+it\in \Omega$ for all $t>0$. Then let $\eta:(0,+\infty)\to \Omega$ be defined as $\eta(t)=\frac{a}{2}+it$. Using standard estimates for the hyperbolic metric (see, {\sl e.g.} \cite[Theorem 3.4]{BCDG} or \cite[Theorem 5.2.2]{BCDbook}), we have for all $1\leq s\leq t$
\[
\ell_\Omega(\eta;[s,t])=\int_s^t\kappa_\Omega(\eta(r);\eta'(r))dr\leq \int_s^t\frac{dr}{\delta_\Omega(\eta(r))}\leq \frac{1}{C}(t-s),
\]
where $C:=\min\{\delta_\Omega(\eta(t)), \,t\in [1,+\infty)\}>0$. 

On the other hand (see, {\sl e.g.} \cite[Proposition 5.2]{BCDG}  or   \cite[Proposition 6.7.2]{BCDbook})
\[
k_\Omega(\frac{a}{2}+is,\frac{a}{2}+it)\geq k_\Lambda(\frac{a}{2}+is,\frac{a}{2}+it)=\frac{\pi}{2a}(t-s).
\]
From these inequalities, it follows that $\eta|_{[1,+\infty)]}$ is a $(\frac{2a}{\pi C},0)$-quasi-geodesic in $\Omega$ (and, by the same argument, it is a quasi-geodesic in $\Omega^\ast$). In particular, $h^{-1}(\eta(t))$ and $(h^\ast)^{-1}(\eta(t))$ converge to $1$ non-tangentially (see, {\sl e.g.} \cite[Remark 3.3]{BCDGZ} or \cite[Corollary 6.3.9]{BCDbook}). From this, we obtain as before that both $\{f^{\circ n}(z)\}$ and $\{\phi_{n}(z)\}$ converge non-tangentially to $1$. 
\end{proof}

Now we can prove Theorem~\ref{Thm:semigroupfication-intro}:

\begin{proof}[Proof of Theorem~\ref{Thm:semigroupfication-intro}]
Statements (1), (2) and (3)  follow directly from the previous considerations, while statement (5) is just Lemma~\ref{Lem:premain}. 

Statement (4) follows from Theorem~\ref{BRFP-model}, since every maximal $\mu$-spirallike sector ({\sl respectively}, maximal vertical strip) for $\Omega$ is clearly a maximal $\mu$-spirallike sector ({\sl resp.}, maximal vertical strip) for $\Omega^\ast$, and vice-versa.
\end{proof}

Finally, we have

\begin{proof}[Proof of Theorem~\ref{Thm:main}]
If $f$ is hyperbolic, the result is known, as we already explained. If $f$ is parabolic, let $(\phi_t)$ be the semigroup-fication of $f$. By Theorem~\ref{Thm:semigroupfication-intro}.(5) and Lemma~\ref{Lem:curve-in}, the statement of the theorem holds if and only if the same statement holds for the semigroup $(\phi_t)$. But for semigroups the result has been proved in  \cite[Theorem 1.1]{BCDGZ}.
\end{proof}

\end{document}